\documentclass[journal,print]{ieeecolor} 
\usepackage{lcsys}
\usepackage{amsmath,amssymb,amsfonts}
\usepackage{algorithmic}
\usepackage{graphicx}
\usepackage{textcomp}
 \usepackage[ruled,vlined]{algorithm2e}
 \usepackage{wrapfig}
\usepackage{dsfont}
\usepackage{float}
\usepackage{textcomp}
\usepackage{parskip}
\usepackage{svg}
\usepackage{float}
\usepackage{lipsum}
\usepackage{stfloats}
\usepackage{comment}
\usepackage{xcolor}
\usepackage{hyperref}
\usepackage{cite}
\usepackage{adjustbox}
\hypersetup{
    colorlinks=true,
    linkcolor=blue,
    filecolor=magenta,      
    urlcolor=cyan,
    citecolor=red
    pdftitle={Overleaf Example},
    pdfpagemode=FullScreen,
    }

\newtheorem{theorem}{Theorem}

\newtheorem{problem}{Problem}

\newcommand{\X}{{\mathcal X}}
\newcommand{\cS}{{\mathcal S}}

\newcommand{\bP}{{\mathbf P}}
\newcommand{\bR}{{\mathbf R}}
\newcommand{\bX}{{\mathbf X}}

\graphicspath{{./},{./figures/}}

\definecolor{jade}{rgb}{0.0, 0.66, 0.42}

\def\BibTeX{{\rm B\kern-.05em{\sc i\kern-.025em b}\kern-.08em
    T\kern-.1667em\lower.7ex\hbox{E}\kern-.125emX}}
\begin{document}
\title{
An excursion onto Schr\"odinger's bridges:
Stochastic flows with spatio-temporal marginals}

\author{Asmaa Eldesoukey, Olga Movilla Miangolarra, and Tryphon T. Georgiou
\thanks{This research has been supported in part by the AFOSR under FA9550-23-1-0096, and ARO under W911NF-22-1-0292.}
\thanks{The authors are with the Department of Mechanical and Aerospace Engineering, University of California, Irvine, CA 92697, USA; Emails: \{aeldesou,omovilla,tryphon\}@uci.edu.}%
}

\maketitle

\begin{abstract}

In a gedanken experiment, in 1931/32, Erwin Schr\"odinger sought to understand how unlikely events can be reconciled with prior laws dictated by the underlying physics. In the process, he posed and solved a celebrated problem that is now named after him -- the {\em Schr\"odinger's bridge problem} (SBP). In this, one seeks to find the ``most likely'' paths, out of all incompatible paths with the prior, that stochastic particles took while transitioning. 
The SBP proved to have yet another interpretation, that of the stochastic optimal control problem to steer diffusive particles so as to match specified marginals --  soft probabilistic constraints. Interestingly, the SBP is convex and can be solved by an efficient iterative algorithm known as the {\em Fortet-Sinkhorn algorithm}. The dual interpretation of the SBP, as an estimation and a control problem, as well as its computational tractability, are at the heart of an ever-expanding range of applications. 
The purpose of the present work is to expand substantially the type of control and estimation problems that can be addressed following Schr\"odinger's dictum, by incorporating termination (killing) of stochastic flows. Specifically, in the context of estimation, we seek the most likely evolution realizing measured spatio-temporal marginals of killed particles. In the context of control, we seek a suitable control action directing the killed process toward spatio-temporal probabilistic constraints. 
To this end, we derive a new {\em Schr\"odinger system} of coupled, in space and time, partial differential equations to construct the solution of the proposed problem. Further, we show that a Fortet-Sinkhorn type of algorithm is, once again, available to attain the associated bridge. A key feature of our framework is that the obtained bridge retains the {\em Markovian} structure in the prior process, and thereby, the corresponding controller takes the form of {\em state feedback}.

\end{abstract}


\section{Introduction}
\label{sec:introduction}

\IEEEPARstart{I}{n two} influential treatises \cite{Sch31,Sch32}, E. Schr\"odinger detailed his thoughts on the time symmetry of the laws of nature (``\"Uber die Umkehrung der Naturgesetze'').  Schr\"odinger's ultimate goal was to understand his namesake equation and the weirdness of the quantum world in classical terms. In the referred works, he began with a gedanken experiment to understand how measured atypical events can be reconciled with a prior probability law.  Amid his perusal, Schr\"odinger single-handedly posed and solved the currently known as the {\em Schr\"odinger's bridge problem} (SBP). 
The problem aims to determine the most likely trajectories of particles, as these transition between states, that are inconsistent with an underlying prior law. Formally, one seeks a posterior distribution on paths that interpolates, i.e. ``bridges'', endpoint marginals obtained empirically. 

In the same work, Schr\"odinger quantified the likelihood of rare events with the {\em relative entropy} from the given prior distribution to the empirical, thereby anticipating the development of the {\em large deviations theory} \cite{dembo2009large}. Furthermore, he discovered a number of now-familiar control concepts, including a system of a Fokker-Planck equation and its adjoint, nonlinearly coupled at the endpoints, that is nowadays referred to as a {\em Schr\"odinger system}.
The optimal solution to the SBP, Schr\"odinger concluded, can be obtained by alternatingly solving the two partial differential equations till convergence. This approach was rigorously proven a few years later by R. Fortet \cite{essid2019traversing}, and the iterative procedure is now known as the {\em Fortet-Sinkhorn algorithm} \cite{chen2021stochastic}.

In hindsight, it seems astonishing that the control theoretic nature of the SBP was not recognized until sixty years later, when, in a masterful work \cite{dai1991stochastic},  P. Dai Pra utilized Flemming's logarithmic transformation to link the SBP with the minimal-energy control problem of steering a stochastic system between two endpoint marginals. Fast forward another twenty years with the thesis of  Y. Chen \cite{chen2016modeling} and several related works \cite{chen2015optimal,chen2015optimalII,chen2018optimal,bakolas2018finite,caluya2021wasserstein,tsiotras}, and the SBP began taking up its rightful place within the stochastic control literature. Indeed, the relation between the SBP and the rapidly developing Monge-Kantorovich transportation theory must not be overlooked.  We refer to \cite{chen2021controlling,chen2021stochastic,chen2021optimal, leonard2014survey} for an overview of current developments. Due to the aforementioned dual interpretation and computational tractability, the SBP offers a wide range of applications in control and neighboring fields,  as in estimation  \cite{eldesoukey2023schr}, physics  \cite{miangolarra2024inferring}, and machine learning \cite{de2021diffusion}.

The present work aims to shed light on a new type of control and estimation problems that can be addressed following Schr\"odinger's rationale. The problem at hand considers stochastic particles that diffuse over a bounded time window while some particles can randomly vanish. In this case, the particles are said to evolve according to a {\em killed process}, see \cite{oksendal2013stochastic}.  Within this setting, the marginals of interest are the particles' initial probability density in space and the spatio-temporal density of the vanishing (killed) particles. 
Our control problem seeks to determine the optimal, in the sense of Schr\"odinger, control action of the stochastic system (as an added drift term) and the required killing rate to satisfy both the initial and spatio-temporal marginal constraints. 
The dual estimation problem likewise seeks the most likely evolution to bridge the same initial and spatio-temporal marginals.

The practical motivation for this work is to provide a framework to tackle control problems under soft spatio-temporal conditioning. 
The spatio-temporal density of killed particles constitutes our data and can represent the mass loss in absorbing media, sediment deposition, termination of stochastic agents upon completion of a task, or ruin in risk processes. 
We aim to {\em infer} the probabilistic model explaining the data and, on the flip side, determine the required added drift (feedback controller) and killing rate to fulfill given spatio-temporal marginal requirements. In this work, we detail a new Schr\"odinger system of coupled partial differential equations allowing for constructing solutions to such problems. Moreover, we present a Fortet-Sinkhorn type of algorithm to attain solutions numerically.

 The paper is structured as follows: In Section \ref{sec:problem}, we lay out the newly proposed problem and its mathematical formulation as a large deviation problem. Section \ref{sec:LD} presents the rationale for establishing the solution to the large deviation problem by deriving a new suitable Schr\"odinger system.  Section~\ref{sec:fluid} discusses the numerical computation of the solution, via a Fortet-Sinkhorn-type iteration that reflects the space-time structure of the problem. Before concluding, a numerical example is provided in Section~\ref{sec:examples}.

\section{Problem Description and Formalism}\label{sec:problem}

The starting point of our formulation can be traced to \cite{chen2022most} and \cite{eldesoukey2023schr}. The first paper introduces a variation of the SBP where losses along trajectories of diffusive particles result in unbalanced masses of the two endpoint marginals. The second paper develops a model for random losses on a discrete-time Markov chain where the marginals are available on stopping times at specific sites. The present work builds on this idea -- to develop a model for controlled stochastic flow susceptible to random killing over a bounded time window with the aim to regulate the spatio-temporal profile of killed particles.
 In this case, the probability measure on paths weighs in on trajectories of possibly different lengths.

For notational and analysis purposes, it is more convenient to extend trajectories beyond killing by freezing the value of the killing state till the end of the time window. 
This can be done in a way that killed particles experience a discontinuity in dynamics. To this end, we consider the particles {\em evolve} in a {\em primary} state space, and if killed, the particles get absorbed to a {\em coffin} state space where dynamics is of vanishing drift and stochastic excitation.
 Accordingly, we define a ``fused'' state space $\mathbb X$ that comprises a primary Euclidean space  $\X = \mathbb R^n$ and a ``coffin''  replica  $\widetilde \X$ where killed particles reside. That is,
\begin{align*}
    \mathbb X := \X \cup \widetilde \X.
\end{align*}%
We also consider time, denoted as $t$, in the interval $[0,1]$ to evaluate probabilities of continuation in $\X$ or absorption in $\widetilde \X$.

The particles evolve in $\X$ according to the It\^o diffusion
\begin{align}\label{eq:diffusion}
    dX_t = b(t,X_t) dt + \sigma(t,X_t) dW_t,
    \end{align}
    where $W_t$ is the standard Wiener process and $b$ and $\sigma$ are uniformly Lipschitz continuous on the time interval $[0,1]$, see \cite{van2007stochastic}. This process is subject to a killing rate $V>0$, which we consider a continuous function in both time and space.  Hence, if $\bX_t$ denotes the extended It\^o process in $\mathbb X$, then 
    \begin{align*}
        d \bX_t = \begin{cases}
            d X_t,  & \text{when }\bX_t \in \X, \\ 0, & \text{when } \bX_t \in \widetilde \X,
        \end{cases}
    \end{align*}
where the transition between the dynamics in $\X$ to that in $\widetilde \X$ takes place due to the killing rate $V$, see Fig.~\ref{fig:pictorial} for a pictorial representation.  As illustrated in the figure, the killed particles' trajectories in the fused space appear as stopped trajectories after the killing instant. 
\begin{figure}[t]
    \centering
    \adjustbox{trim=1.5cm 0cm 0cm 0.5cm}{
    \includegraphics[width= 0.7\columnwidth]{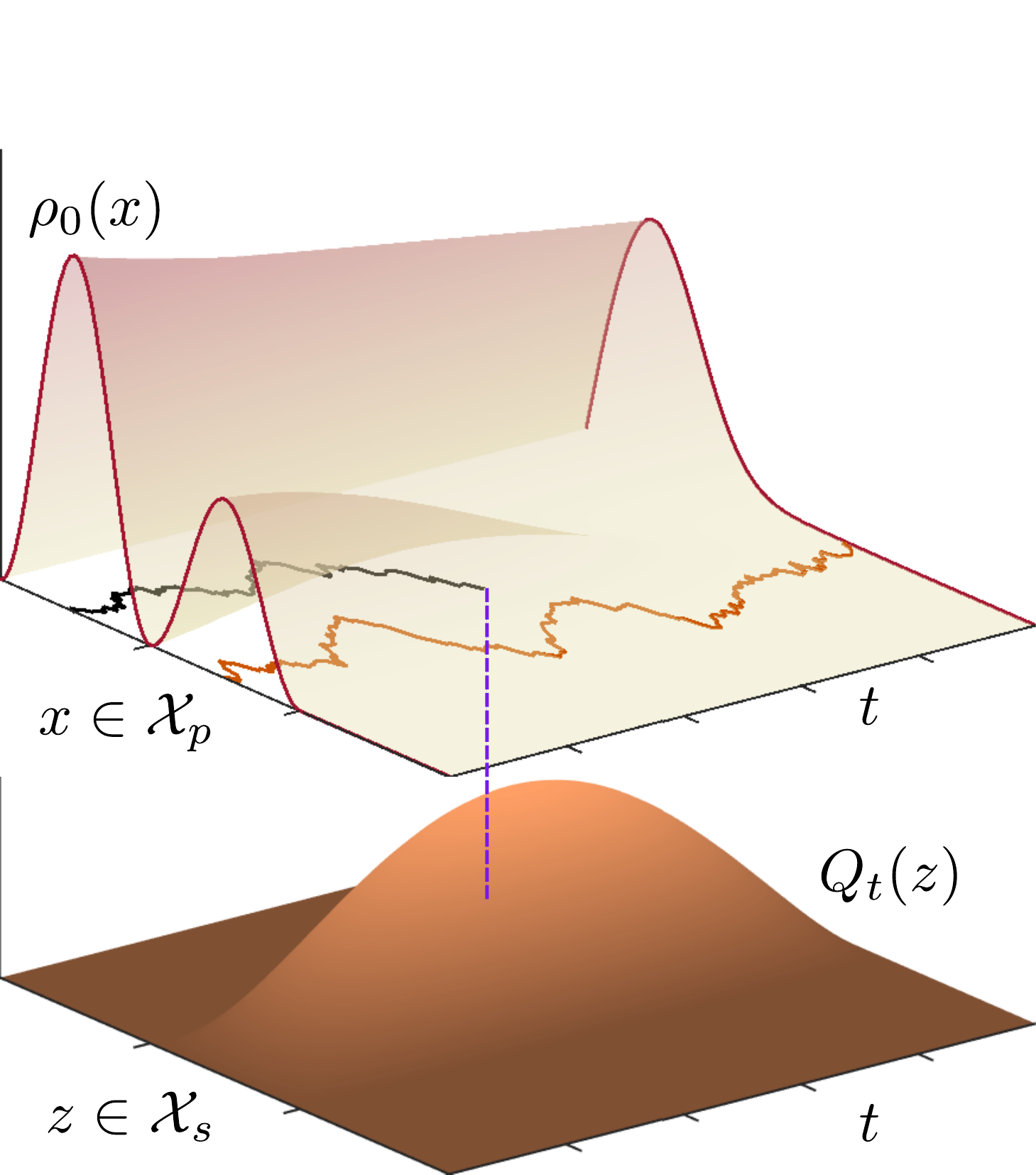}}
    \caption{Illustration of diffusion with losses: The top sub-figure shows two sample paths of particles initially in the primary space $\X$ (a path for a surviving particle and another for a killed one) with a drawing of the initial marginal $\rho_0$ and also of the final marginal of surviving particles. The bottom sub-figure shows the second segment of the killed path and the spatio-temporal marginal $Q_t$ of killed particles in the coffin space $\widetilde \X$.}
    \label{fig:pictorial}
\end{figure}

The space of sample paths will be denoted by $\Omega$,  namely $\Omega := \mathbf D([0,1], \mathbb X)$ where $\mathbf D$ is the {\em Skorokhod space} over $\mathbb X$. The paths can be metrized by the {\em Skorokhod metric} \cite{rogers2000diffusions,pollard1984skorohod}, where curves are compared by how much they need to be perturbed in both the domain (i.e. the time axis) and range, to match. We define $\mathcal P(\Omega)$ as the set of probability measures on $\Omega$ and take $\bR \in \mathcal P (\Omega)$ that corresponds to $\bX_t$ as the prior law.  Let $\bR_t$ denote the one-time marginals of $\bR$, i.e. $\bR_t$ defines ${\rm Law} (\bX_t)$, and $R_t$ {\em the restriction} of $\bR_t$ to $\X$. Then, from \eqref{eq:diffusion}, $R_t$  weakly satisfies
  \begin{align}\label{eq:FP}
    \partial_t R_t = -  \nabla \cdot (b R_t) + \frac{1}{2}  \sum_{i,j=1}^n  \frac{\partial^2 (a_{i,j} R_t)}{\partial x_i \partial x_j} -V R_t,
\end{align}
for the typical assumptions: $a := \sigma \sigma'$ is uniformly positive definite\footnote{$\sigma'$ denotes the transpose of $\sigma$.} and  $a, b, V$ are bounded and $\mathcal C^{1,2}$, see \cite{evans2022partial}. Equation \eqref{eq:FP} is the Kolmogorov forward (Fokker-Planck) equation with a killing rate $V$. Equivalently, we say
\begin{align}\label{eq:r}
    R_t(\cdot) = \int r(0,x,t,\cdot) R_0(x) dx,
\end{align}
where $r(0,x,t,y)$ denotes the Markov transition kernel from $x \in \X$ at time $0$ to $y \in \X$ at time $t$. Here, we also refer to \cite{leonard2022feynman} for an interesting account of the original SBP, diffusion with losses as in \eqref{eq:FP}, and the Feynman-Kac formula.

In this letter, we consider the control problem to specify the optimal update for both the drift and the killing rate of $X_t$, so as to ensure suitable absorption-related marginal at different times. In turn, the data for our problem is the spatio-temporal marginal of the accumulated particles in the coffin space, denoted by $Q_t$, together with the initial probability density $\rho_0$. 
 Assuming the prior law $\bR$ is inconsistent with the data, from the estimation perspective, we seek to update the prior law with an optimal posterior law in $\mathcal P(\Omega)$ that reconciles with the data. Any candidate (i.e. consistent) posterior law is denoted herein as $\bP$ with $\bP_t$ and $P_t$ being its one-time marginals and their restriction to $\X$, respectively. The optimal posterior, in the sense of Schr\"odinger, is the measure in $\mathcal P(\Omega)$ closest to $\bR$ apropos to a relative entropy functional {\em and} matches the data. For that, we propose the following problem.

\begin{problem} \label{prob:CTCSBrid}
Given a prior Markov probability measure $\bR \in \mathcal P(\Omega)$ with one-time marginals restricted to $\X$ evolving according to \eqref{eq:FP}, a probability density $\rho_0$ with support in $\X$ and a spatio-temporal marginal $Q_t$ with support in $\widetilde \X \times [0,1]$, determine 
\begin{subequations}\label{eq:problem1}
\begin{align} \label{eq:likelihoodfn}
     \bP^\star:= \arg \min_{\bP \in \mathcal P (\Omega), \bP \ll \bR} & \mathbb D (\bP \parallel \bR)  := \int_{\Omega} d \bP \log \frac{d \bP }{d \bR}, \\
      \mbox{ subject to } \bP_0(x)  = &\  \rho_0(x),  \ \text{for } x \in \X, \label{eq:massconst} 
\\
      \label{eq:Qtconst}
      \mbox{and }   \bP_t(z)  = &\ Q_t(z),  \ \text{for } z \in \widetilde \X, \, t\in[0,1].
   \end{align}
\end{subequations}
\end{problem}
The objective functional in \eqref{eq:likelihoodfn} is the {\em relative entropy divergence} from $\bR$ to $\bP$. It is strictly convex and bounded when $\bP$ is absolutely continuous\footnote{Throughout, as usual, $0 \log \frac{0}{0} := 0$.} with respect to $\bR$, a relation denoted by $\bR \gg \bP$ \cite{kullback1951information,cover1999elements}. The specification in \eqref{eq:massconst} amounts to the simplifying assumption that no particles are killed initially since $\int_{\X} \rho_0(x) dx =1$; thus, for consistency, we take $Q_{t=0}(\cdot) = 0$. 
Throughout, we let $\alpha(t,x) V(t,x)$ be the killing rate corresponding to $\bP^\star$, with $\alpha \geq 0$ the optimal rescaling over the prior killing rate $V$.
This rescaling must be such that the density of absorbed particles $\alpha(t,x)  V(t,x)  P_t(x)$ coincides with $Q(t,z)$, for $z=x$.

{A practical motivation to Problem \ref{prob:CTCSBrid} is drawn by considering $\bX_t$ as modeling the stochastic flow of particles that drift along with the chance of being deposited at some location $z$ at time $t$. A deposited ``mass landscape'' $Q_t(z)$ that does not match what was expected is observed. Thus, we seek to infer dynamics together with an updated deposition rate so as to restore consistency with the recorded data.
A control theoretic significance for Problem \ref{prob:CTCSBrid} stems from applications involving agents that obey stochastic dynamics with a rate of termination (e.g. due to the completion of a task); the updated drift represents control input whereas the killing rate $\alpha V$ represents the updated rate of termination.

\section{Establishing the Solution} \label{sec:LD}
In this section, we derive the Schr\"odinger system corresponding to Problem \ref{prob:CTCSBrid}.
We start by noting that the space of sample path $\Omega$ decomposes such that
\begin{align*}
     \Omega = \cS \cup \widetilde \cS,
\end{align*}
where the symbol $\cS$ denotes the space of sample paths of {\rm surviving} particles over the interval $[0,1]$, and $\widetilde \cS$ its complement.

 We first define
\begin{align*}
    \bP_{0,x}^{1,y}(\cdot) &:= \bP(.|\bX_0 =x,\bX_1 =y),\\
    \widetilde \bP_{0,x}^{t,z}(\cdot) &:= \bP(.| \bX_0 =x, \bX_T =z, T=t),
\end{align*}
where $T$ denotes the random variable for killing time and $x,y \in \X, z \in \widetilde \X, t \in [0,1]$. The conditioning $\bP_{0,x}^{1,y}$ ($\widetilde \bP_{0,x}^{t,z}$) describes the disintegration of $\bP$ with respect to the initial position and final position in $\X$ (initial position in $\X$, vanishing position in $\widetilde \X$ and the vanishing instant). Then, we say
\begin{align} 
\begin{split}\label{eq:fulllaw}
      &\bP(\cdot) = \begin{cases}\displaystyle
       \quad \int_{\X^2} \bP_{0,x}^{1,y}(\cdot) \pi_{xy}(x,y) dxdy,  \ &\text{on } \cS, \\[8 pt]\displaystyle
 \ \int_{\X \times \widetilde \X \times [0,1] } \! \! \!
    \widetilde \bP_{0,x}^{t,z}(\cdot) \pi_{xzt}(x,z,t)  dxdzdt, \ &\text{on } \widetilde \cS,
    \end{cases}
\end{split}\\[10pt]
 & \text{where } \ \ \ \ \ \ \  \pi_{xy}(x,y) dx dy := d \bP(\bX_0 =x, \bX_1 =y), \nonumber \\
  & \text{and } \ \ \ \ \  \pi_{xzt}(x,z,t) dx dz dt :=  d \bP(\bX_0 =x, \bX_T =z, T =t). \nonumber 
\end{align}
The symbols $\pi_{xy}$ and $\pi_{xzt}$  denote joint densities on $\X^2$ and $\X \times \widetilde \X \times [0,1]$, respectively, that are commonly referred to as {\em couplings}; for convenience,   we have assumed that these are absolutely continuous with respect to the Lebesgue measure.  We refer to \cite{1978iii} for more details on the disintegration theorem.

In a similar manner, we let the disintegrated counterparts for the prior $\bR$ be $\bR_{0,x}^{1,y}$ and $\widetilde \bR_{0,x}^{t,z}$ with the couplings $\rho_{xy}$ on $\X^2$ and $\rho_{xzt}$ on $\X \times \widetilde \X \times [0,1]$. 
Then, by applying the additive property of relative entropy \cite{leonard2014survey} to \eqref{eq:fulllaw}, we get
\begin{align} 
\! \! \! \! \! & \mathbb D (\bP \parallel \bR) = \int_{\X^2} \pi_{xy}(x,y) \log \frac{ \pi_{xy}(x,y)}{ \rho_{xy}(x,y)} dx dy \  \nonumber \\
    &  + \int_{\X^2} \pi_{xy}(x,y)  \bigg[ \mathbb D(\bP_{0,x}^{1,y} \parallel \bR_{0,x}^{1,y}) \bigg]   dx dy  \nonumber \\
    & + \int_{\X\times \widetilde{\X} \times[0,1]} \pi_{xzt}(x,z,t) \log \frac{ \pi_{xzt}(x,z,t)}{ \rho_{xzt}(x,z,t)}   dx dz dt \  \nonumber \\
    & + \int_{\X\times \widetilde{\X} \times[0,1]} \pi_{xzt}(x,z,t)  \bigg[ \mathbb D(\widetilde \bP_{0,x}^{t,z} \parallel \widetilde \bR_{0,x}^{t,z}) \bigg]   dx dz dt. \nonumber \\
    & = \mathbb D(\pi_{xy} \parallel \rho_{xy}) + \mathbb E_{\pi_{xy}} \bigg[ \mathbb D(\bP_{0,x}^{1,y} \parallel \bR_{0,x}^{1,y}) \bigg] \nonumber  \\ 
    &+  \mathbb D(\pi_{xzt} \parallel \rho_{xzt}) + \mathbb E_{\pi_{xzt}} \bigg[ \mathbb D(\widetilde \bP_{0,x}^{t,z} \parallel \widetilde \bR_{0,x}^{t,z}) \bigg]. 
    \label{eq:disint}
\end{align}
We note that the second and fourth terms on the right-hand side of equation \eqref{eq:disint} are minimal and equal to zero if and only if \[\bP_{0,x}^{1,y}=\bR_{0,x}^{1,y}, \  \widetilde \bP_{0,x}^{t,z}=\widetilde \bR_{0,x}^{t,z}.\]
That is when the prior and posterior have identical ``pinned'' bridges,  Problem \ref{prob:CTCSBrid} can be parameterized in terms of the couplings only.
Therefore, 
\begin{align*}
    \mathbb D (\bP^\star \parallel \bR) =  \mathbb D(\pi^\star_{xy} \parallel \rho_{xy}) + \mathbb D(\pi^\star_{xzt} \parallel \rho_{xzt}),
\end{align*}
 where the couplings $\pi^\star_{xy},\pi^\star_{xzt}$ can be obtained as the solutions to the following problem.
 \begin{problem} \label{prob:JP1}
    Minimize
    \begin{subequations}\label{eq:constraints}
        \begin{align}
        & \mathbb D(\pi_{xy} \parallel \rho_{xy}) + \mathbb D(\pi_{xzt} \parallel \rho_{xzt}), \nonumber \text{ subject to } \\
       &    \int_{\X} \pi_{xy}(x,y) dy + \int_0^1 \! \! \! \int_{\widetilde \X} \pi_{xzt}(x,z,t) dz  dt = \rho_0(x), \label{eq:space-const}  \\
        \label{eq:spacetime-const}
      &  \mbox{and }  \int_{\X} \pi_{xzt}(x,z,t) dx = Q_t(z),
    \end{align}
    \end{subequations}
    with $\pi_{xy} \ll \rho_{xy} , \pi_{xzt} \ll \rho_{xzt}$.
    \end{problem}
    
    The augmented Lagrangian of Problem \ref{prob:JP1} is
 \begin{align*}
     &\mathcal L =  \mathbb D(\pi_{xy} \parallel \rho_{xy}) + \mathbb D(\pi_{xzt} \parallel \rho_{xzt}) +\\
     & \! \int_{\X} \! \! \! \mu(x) \! \bigg[  \int_{\X} \pi_{xy}(x,y) dy 
      + \! \int_0^1 \! \! \! \int_{\widetilde \X} \pi_{xzt}(x,z,t) dz  dt - \!\rho_0(x)\bigg] \! dx  \\
     &+ \int_0^1 \! \! \! \int_{\widetilde \X} \eta(t,z) \bigg[\int_{\X} \pi_{xzt}(x,z,t) dx -Q_t(z) \bigg] dz dt,
 \end{align*}
 where $\mu, \eta$ are Lagrange multipliers enforcing \eqref{eq:space-const} and \eqref{eq:spacetime-const}, respectively.
 From the first-order optimality conditions
   \begin{align*}
   0 = &1+ \log \frac{\pi^\star_{xy}
   (x,y)}{\rho_{xy}(x,y)} + \mu(x), \mbox{ and} \\
  0=&  1+ \log \frac{\pi^\star_{xzt}(x,z,t)}{\rho_{xzt}(x,z,t)} + \mu(x) + \eta(t,z),
\end{align*}
we deduce that the minimizer of Problem \ref{prob:JP1} satisfies
    \begin{subequations} \label{eq:opt1}
         \begin{align}
        \pi^\star_{xy}(x,y) &=\rho_{xy}(x,y) f(x), \label{eq:opt1a}\\
        \pi^\star_{xzt}(x,z,t) &= \rho_{xzt}(x,z,t) f(x) \Lambda(t,z),\label{eq:opt1b}
    \end{align}
    \end{subequations}%
  for $f(x) = \exp(-1 - \mu(x))$, and $\Lambda(t,z) = \exp(-\eta(t,z))$. Substituting \eqref{eq:opt1} into \eqref{eq:space-const} we obtain that
  \begin{subequations} \label{eq:Sch1}
  \begin{align}
  \varphi(0,x)\hat \varphi(0,x)=\rho_0(x),
      \label{eq:coupl1}
  \end{align}
 where we took
\begin{align}
   &\varphi(0,x) =  \int\limits_{\X}  \frac{\rho_{xy}(x,y)}{R_0(x)} dy \ + \int\limits_0^1 \! \!  \int\limits_{\widetilde \X} \Lambda(t,z) \frac{\rho_{xzt} (x,z,t)}{R_0(x)} dz dt, \nonumber
\end{align}
and $\hat \varphi(0,x) =  R_0(x)f(x)$.
Inspecting closer we see that
\begin{align}\nonumber
 \varphi(0,x) =&   \int_{\X}   r(0,x,1,y) dy \\
    & + \int_0^1   \int_{\X} \Lambda(t,z) V(t,z) r(0,x,t,z) dz dt, \label{eq:BK} \tag{\rm FS1}
    \end{align}
with $r$ the transition kernel of the prior (see \eqref{eq:FP} and \eqref{eq:r}).  Equation \eqref{eq:BK}, and subsequent similarly labeled equations, serve as steps of an {\em iterative numerical algorithm} ({\em Fortet-Sinkhorn}) to determine unknowns $\varphi,\hat\varphi,\Lambda,\hat \Lambda$ as explained below and in the next section or serve in the next proof.

By inserting $\varphi(1,\cdot)=1$ into the first integral on the right hand side of \eqref{eq:BK}, we see that
 $\varphi(0,\cdot)$ can be obtained by solving the Kolmogorov backward equation
\begin{align}
      \partial_t \varphi =- b \cdot \nabla \varphi - \frac{1}{2} \! \! \sum_{i,j=1}^n \! \! a_{i,j} \frac{\partial^2 \varphi}{\partial x_i  \partial x_j} +V \varphi -  V \Lambda,  \label{eq:backdiff}
\end{align}
with terminal condition $\varphi(1,\cdot) = 1$.
Similarly, by substituting \eqref{eq:opt1b} into \eqref{eq:spacetime-const}, we obtain that
\begin{align}
 \Lambda(t,z) \hat \Lambda(t,z) &=Q_t(z), \text{where }  \label{eq:coupl2}\\
\hat \Lambda(t,z) &= \int_{\X} \rho_{xzt} (x,z,t) f(x)  dx \nonumber \\
& =  \int_{\X}  V(t,z) r(0,x,t,z)   \hat \varphi(0,x) dx. \label {eq:FK} \tag{\rm FS2}
\end{align}
This expression motivates defining 
\begin{align} 
    \hat \varphi(t,\cdot) =  \int_{\X} r(0,x,t,\cdot)  \hat \varphi(0,x) dx, \label{eq:hatphik} \tag{FS3}
\end{align}
that satisfies the Kolmogorov forward equation
\begin{align}
    \partial_t \hat \varphi &= -  \nabla \cdot (b \hat \varphi) + \frac{1}{2}  \sum_{i,j=1}^n  \frac{\partial^2 ( a_{i,j} \hat \varphi)}{\partial x_i \partial x_j} -V \hat \varphi,  \label{eq:fordiff}\\
\text{and }   \hat \Lambda(t,z)  &= V(t,z)  \hat \varphi(t,z) . \label{eq:lambdahat}
\end{align}
\end{subequations}
\indent Equations (\ref{eq:coupl1}-\ref{eq:lambdahat}) constitute the Schr\"odinger system associated with Problems \ref{prob:CTCSBrid} and \ref{prob:JP1}. 
In contrast to the classical Schr\"odinger system \cite{chen2016entropic}, 
the partial differential equations \eqref{eq:backdiff} and \eqref{eq:fordiff} are coupled at all times via \eqref{eq:coupl2} and \eqref{eq:lambdahat}. The construction of numerical solutions to this system will be discussed in the next section. We now highlight that the Schr\"odinger system provides the solution to Problem 1.

 \begin{theorem}\label{thm:1}  Assume that Problem \ref{prob:CTCSBrid} is feasible. Then:
\begin{itemize}
   \item[i)] the system (\ref{eq:coupl1}-\ref{eq:lambdahat}) has a unique solution\footnote{This solution is unique modulo a scaling factor, in that,  for any $\kappa>0$,  $(\kappa \varphi, \kappa^{-1}\hat\varphi,\kappa\Lambda,\kappa^{-1}\hat\Lambda)$ is also a solution and there are no others.} $(\varphi, \hat\varphi,\Lambda,\hat\Lambda)$,
    \item[ii)] $\bP^\star$ in Problem \ref{prob:CTCSBrid} is the law of the Markov process 
\begin{align*} 
    dX_t = [b(t,X_t)+ \sigma(t,X_t) u(t,X_t)]dt + \sigma(t,X_t) d W_t,
\end{align*}
with the killing rate $(\Lambda/\varphi) V$ and the term $\sigma u$ representing the control input where $u = \sigma' \nabla \log \varphi$,
\item[iii)]  the one-time marginals for $\bP^\star$ restricted to $\X$ are
\begin{align}\label{eq:Pt}
    P_t(\cdot) = \varphi(t,\cdot) \hat \varphi(t,\cdot) ,
\end{align}
and  satisfy
\begin{align}\label{eq:post1}
    \partial_t P_t = &  - \nabla \cdot((b + \sigma u)P_t) + \frac{1}{2} \! \! \sum_{i,j=1}^n \! \! \frac{\partial^2  (a_{i,j} P_t)}{\partial x_i \partial x_j} - \alpha V P_t,
\end{align}
 with $\alpha = \Lambda /\varphi$.
\end{itemize}
\end{theorem}

\begin{proof}
The functional in Problem \ref{prob:CTCSBrid} is strictly convex and the constraints are linear. Hence, feasibility implies the existence of (unique) suitable functions $f,\Lambda$ that satisfy equations~\eqref{eq:opt1}. Integrating \eqref{eq:backdiff} gives $\varphi$, which in turn gives $\hat\varphi(0,\cdot)$ from \eqref{eq:coupl1} and hence $\hat\varphi$ from \eqref{eq:fordiff} (also, \eqref{eq:hatphik}), and finally, $\hat\Lambda$ from \eqref{eq:lambdahat}. This establishes statement i).

In the discussion leading to the theorem we have seen that\footnote{Information about the type of the paths (killed or not) is already encoded in the pinned bridges since $\bR_{0,x}^{1,y}(A \subseteq \widetilde{\cS}) =0$ and $\widetilde \bR_{0,x}^{t,z}(A \subseteq \cS) =0$.}
\begin{align*}
  \bP^\star(\cdot)  &=  \int_{\X^2}\bR_{0,x}^{1,y}(\cdot) \pi^\star_{xy}(x,y) dxdy\,  \\
  &\ + \!\int_{\X \times \widetilde \X \times [0,1] }\!\!\!\!
    \widetilde \bR_{0,x}^{t,z}(\cdot) \pi^\star_{xzt}(x,z,t)  dxdzdt,
\end{align*}
where $\pi^\star_{xy}$ and $\pi^\star_{xzt}$ are as in \eqref{eq:opt1}.
From \eqref{eq:opt1} we also observe that $\bP^\star$ is Markov since the prior is Markov and the correction in \eqref{eq:opt1} amounts to multiplicative scaling at fixed points in time. We will return and complete the proof of ii), displaying the terms in the equation, after we first establish iii).

Using the form of $\bP^\star$ given above, we show that the one-time marginals restricted to $\X$ are as in~\eqref{eq:post1}. Indeed, 
\begin{align*}
  &  P_\tau(x_\tau) dx_\tau \!= \! \!\int\limits_{\X^2}   \! \! d\bR(\bX_\tau = x_\tau| \bX_0 = x, \bX_1 = y) \pi^\star_{xy}(x,y) dx dy  \\
   & + \int_{\X \times \widetilde \X \times (\tau,1]}  \!  \big\{ d \bR(\bX_\tau = x_\tau | \bX_0 = x, \bX_T = z,T=t) \\ & \hspace{75pt} \times \pi^\star_{xzt}(x,z,t) \big\} dx dz dt ,
\end{align*}
 at any time $\tau$, and for any $x_\tau \in \X$. Thus, the value $P_\tau(x_\tau)$ results from two contributions, first from the paths that survive over the window $[0,1]$ while passing through $x_\tau$, and second, from paths that terminate at any time in $(\tau,1]$. From the Markovianity of the prior, we know
\begin{align*}
    & d\bR(\bX_\tau = x_\tau| \bX_0 = x, \bX_1 = y)  = \\ 
    & \hspace{120pt}  \frac{r(0,x,\tau,x_\tau) r(\tau,x_\tau,1,y)}{r(0,x,1,y)} dx_\tau ,\\
    & d \bR(\bX_\tau = x_\tau | \bX_0 = x, \bX_T = z, T =t)  = \\ 
    & \hspace{120pt} \frac{r(0,x,\tau,x_\tau) r(\tau,x_\tau,t,z)}{r(0,x,t,z)} dx_\tau.
\end{align*}
These expressions, together with \eqref{eq:opt1}, that is,
\begin{align*}
    \pi^\star_{xy}(x,y) &=  r(0,x,1,y) R_0(x) f(x)= r(0,x,1,y) \hat \varphi(0,x), \\
    \pi^\star_{xzt}(x,z,t) &=  V(t,z) r(0,x,t,z) \hat \varphi(0,x)  \Lambda(t,z),  
\end{align*}
lead to
\begin{align*}
    &P_\tau(x_\tau) = \int\displaylimits_{\X^2} r(0,x,\tau,x_\tau) r(\tau,x_\tau,1,y) \hat \varphi(0,x) dx dy  \ + \\
    & \! \! \!  \int\displaylimits_{\X^2 \times (\tau,1]} \! \! \! \Lambda(t,z) V(t,z) r(0,x,\tau,x_\tau) r(\tau,x_\tau,t,z)   \hat \varphi(0,x)   dx dz dt.
\end{align*}
Factoring out $\hat\varphi$ (see \eqref{eq:hatphik}) leads to $P_\tau(\cdot) =  \varphi(\tau,\cdot) \hat \varphi(\tau,\cdot)$,
and from \eqref{eq:backdiff}, we have $\varphi(\tau,x_\tau)$  equals
\begin{align*} 
    \int_{\X}  \! \!   r(\tau,x_\tau,1,y) dy +  \int_\tau^1 \! \!   \int_{\X}  \Lambda(t,z) V(t,z) r(\tau,x_\tau,t,z) dz dt. 
\end{align*}
Having established \eqref{eq:Pt}, we can verify \eqref{eq:post1} by substituting \eqref{eq:backdiff} and \eqref{eq:fordiff} in
$ \partial_t P_t =  \varphi \partial_t \hat \varphi  +\hat \varphi \partial_t \varphi .$ Finally, from the generator in \eqref{eq:post1}, we can read off the added drift (feedback controller) and killing rate for the stochastic differential equation in ii).
\end{proof}

 \section{The Fortet-Sinkhorn Algorithm} \label{sec:fluid}
 We now present an algorithm that provides a numerical solution to the Schr\"odinger system \eqref{eq:Sch1}.  The solution can be obtained from the limit of a Fortet-Sinkhorn-type iteration:
 \begin{align}
\begin{split} \label{eq:sink1}
     \!\!\! \varphi(0,\cdot)   \mathop{\mapsto}_{\eqref{eq:coupl1}} \hat \varphi(0,\cdot) \mathop{\mapsto}_{\eqref{eq:FK}}  \hat \Lambda \mathop{\mapsto}_{\eqref{eq:coupl2}} \Lambda
      \mathop{\mapsto}_{\eqref{eq:BK}} (\varphi(0,\cdot))_{\rm next}. 
\end{split}
\end{align}
This is stated next.

\begin{theorem} 
Under the assumptions in Problem \ref{prob:CTCSBrid} and starting from a positive real-valued function $\varphi(0,\cdot)$, the iteration \eqref{eq:sink1} converges to a unique fixed point, i.e., to a quadruple $(\varphi(0,\cdot),\hat \varphi(0,\cdot), \Lambda,\hat \Lambda)$ leading to the solution of the Schr\"odinger system \eqref{eq:Sch1}.
\end{theorem}
\begin{proof}
The claim follows from standard arguments used to prove convergence of the classical Fortet-Sinkhorn iteration --the cycling through \eqref{eq:sink1} is strictly contractive, see \cite{chen2022most,chen2016entropic,georgiou2015positive}. Specifically, the composition of maps in \eqref{eq:sink1} is a strict contraction on the cone, excluding the apex, of positive real-valued functions denoted as $\mathcal K^+$ with respect to the Hilbert-projective metric.
The steps follow exactly those in \cite{chen2022most} and \cite{chen2016entropic}, where the map $\varphi(0,\cdot) \to (\varphi(0,\cdot))_{\rm next}$ decomposes into two isometries on $\mathcal K^+$ and two strictly linear contractive maps on $\mathcal K^+$ by virtue of Birkhoff's theorem. Thus, the iteration converges to a unique fixed point  $(\varphi(0,\cdot), \hat \varphi(0,\cdot), \Lambda, \hat \Lambda)$.
\end{proof}

\section{Example} \label{sec:examples}
\begin{figure}[t]
    \centering
    \includegraphics[width = 0.6\columnwidth]{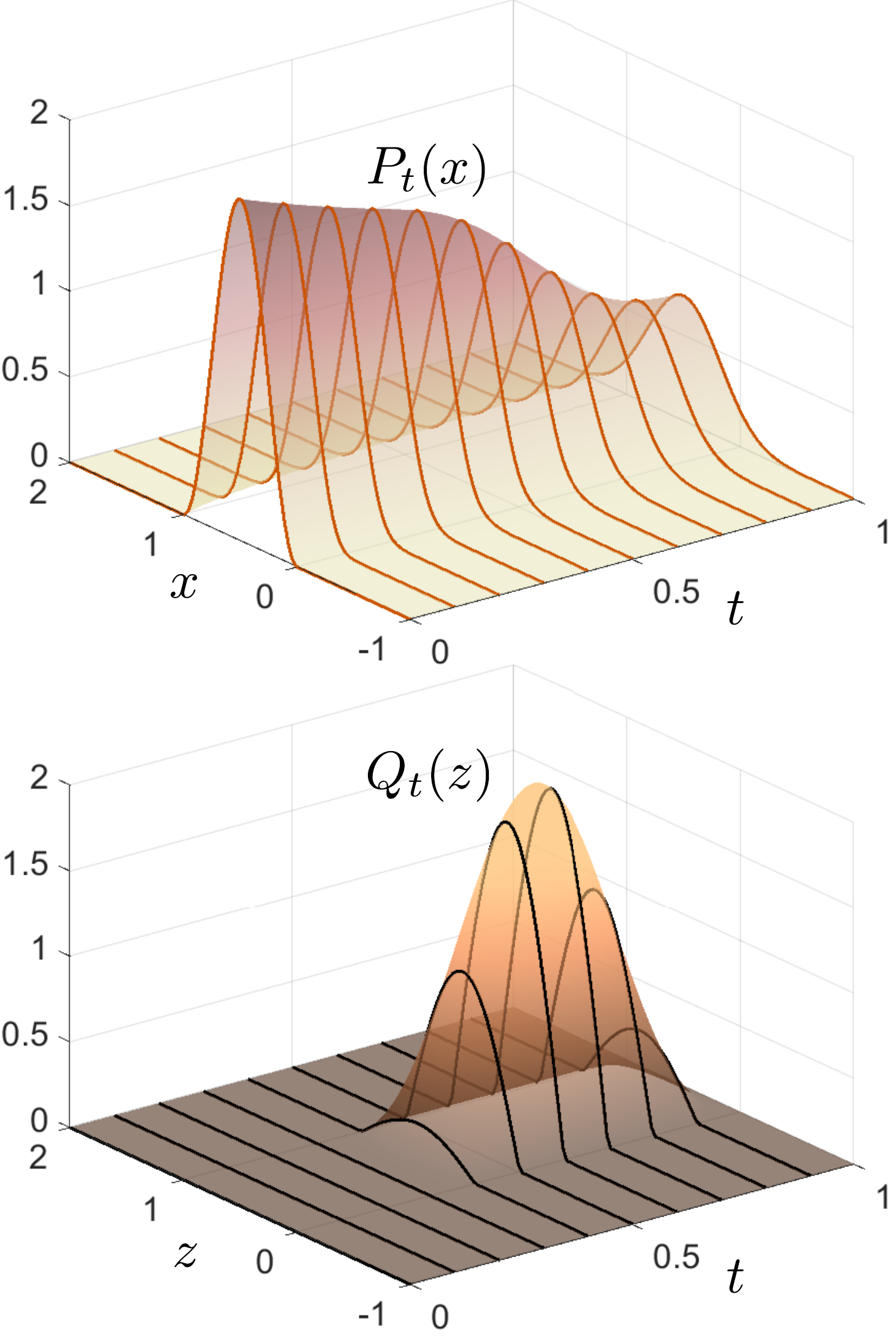}
    \caption{Top: One-time marginals of $\bP^\star$ restricted to the primary space. Bottom: One-time marginals of $Q_t$ (black), the spatio-temporal marginal of the particles in the coffin space as obtained from the Fortet-Sinkhorn iteration \eqref{eq:sink1} (shaded surface).}
    \label{fig: sigma}
\end{figure}
We demonstrate the above algorithm with a numerical example. We take as prior law the one corresponding to the diffusion process $dX_t = \frac{1}{4} d W_t,$
with killing $V(t,x)=0.3$, and $t \in [0,1]$. We seek to match the marginal constraints 
\begin{align*}
\rho_0(x) &= \left(1 -\cos(2 \pi x)\right) \mathds 1_{\{[0,1]\}}(x),\\
    Q_t(z) &= 
          \sin(\pi z)\left(1 - \cos(3 \pi t -  \pi)\right)\mathds 1_{\{[0,1]\}}(z)\mathds 1_{\{[\frac{1}{3},1]\}}(t),
\end{align*}
where $\mathds{1}_{\{A\}}$ denotes the indicator function of the set $A$.
This spatio-temporal marginal accounts for about $42 \%$ of the total initial mass while ensuring delay in the mass absorption till $t=1/3$ 
and is shown in Fig.~\ref{fig: sigma} (bottom sub-figure, with black curves delineating the one-time marginals of $Q_t$).

By direct application of the Fortet-Sinkhorn algorithm~\eqref{eq:sink1} we obtain the posterior law $\bP^\star$ that indeed satisfies the constraints; Fig.~\ref{fig: sigma} (top) shows the initial (agreeing with $\rho_0$) and subsequent one-time marginals of $\bP^\star$ in the primary space with the mass loss taking place over $[\frac{1}{3},1]$ (time), as specified. The shaded surface in Fig.~\ref{fig: sigma} (bottom) is the spatio-temporal marginal obtained from iteration \eqref{eq:sink1} in agreement with $Q_t$.

\section{Conclusions}
This letter focuses on the control problem of regulating stochastic systems so as to match specified spatio-temporal data of associated killed sample paths. 
Killing may model the completion of a task, or mass being absorbed or deposited along trajectories.
The formulation herein is envisioned as a first step towards a variety of new stochastic control problems with data and specifications cast in the form of soft (probabilistic) conditioning. 
To this end, Schr\"odinger's paradigm, anchored in the large deviations rationale, proved versatile and computationally amenable for such purposes. 
A future direction of great interest is to study the noiseless limit due to vanishing stochastic excitation in \eqref{eq:diffusion} while meeting spatio-temporal marginal constraints in line with an analogous optimal mass transport theory.

\bibliographystyle{ieeetr}
\bibliography{main.bib}

\begin{thebibliography}{10}

\bibitem{Sch31}
E.~Schr{\"o}dinger, ``{\"U}ber die {U}mkehrung der {N}aturgesetze,'' {\em {S}itzungsberichte der {P}reuss {A}kad. {W}issen. {P}hys. {M}ath. {K}lasse, {S}onderausgabe}, vol.~IX, pp.~144--153, 1931.

\bibitem{Sch32}
E.~Schr{\"o}dinger, ``Sur la th{\'e}orie relativiste de l'{\'e}lectron et l'interpr{\'e}tation de la m{\'e}canique quantique,'' in {\em Annales de l'institut Henri Poincar{\'e}}, vol.~2(4), pp.~269--310, Presses universitaires de France, 1932.

\bibitem{dembo2009large}
A.~Dembo and O.~Zeitouni, {\em Large deviations techniques and applications}, vol.~38.
\newblock Springer Science \& Business Media, 2009.

\bibitem{essid2019traversing}
M.~Essid and M.~Pavon, ``Traversing the {S}chr{\"o}dinger bridge strait: Robert {F}ortet’s marvelous proof redux,'' {\em Journal of Optimization Theory and Applications}, vol.~181, no.~1, pp.~23--60, 2019.

\bibitem{chen2021stochastic}
Y.~Chen, T.~T. Georgiou, and M.~Pavon, ``Stochastic control liaisons: {R}ichard {S}inkhorn meets {G}aspard {M}onge on a {S}chr\"odinger bridge,'' {\em SIAM Review}, vol.~63, no.~2, pp.~249--313, 2021.

\bibitem{dai1991stochastic}
P.~Dai~Pra, ``A stochastic control approach to reciprocal diffusion processes,'' {\em Applied Mathematics and Optimization}, vol.~23, no.~1, pp.~313--329, 1991.

\bibitem{chen2016modeling}
Y.~Chen, {\em Modeling and control of collective dynamics: From {S}chr{\"o}dinger bridges to optimal mass transport}.
\newblock PhD thesis, University of Minnesota, 2016.

\bibitem{chen2015optimal}
Y.~Chen, T.~T. Georgiou, and M.~Pavon, ``Optimal steering of a linear stochastic system to a final probability distribution, {P}art {I},'' {\em IEEE Transactions on Automatic Control}, vol.~61, no.~5, pp.~1158--1169, 2015.

\bibitem{chen2015optimalII}
Y.~Chen, T.~T. Georgiou, and M.~Pavon, ``Optimal steering of a linear stochastic system to a final probability distribution, {P}art {II},'' {\em IEEE Transactions on Automatic Control}, vol.~61, no.~5, pp.~1170--1180, 2015.

\bibitem{chen2018optimal}
Y.~Chen, T.~T. Georgiou, and M.~Pavon, ``Optimal steering of a linear stochastic system to a final probability distribution, {P}art {III},'' {\em IEEE Transactions on Automatic Control}, vol.~63, no.~9, pp.~3112--3118, 2018.

\bibitem{bakolas2018finite}
E.~Bakolas, ``Finite-horizon covariance control for discrete-time stochastic linear systems subject to input constraints,'' {\em Automatica}, vol.~91, pp.~61--68, 2018.

\bibitem{caluya2021wasserstein}
K.~F. Caluya and A.~Halder, ``Wasserstein proximal algorithms for the {S}chr{\"o}dinger bridge problem: Density control with nonlinear drift,'' {\em IEEE Transactions on Automatic Control}, vol.~67, no.~3, pp.~1163--1178, 2021.

\bibitem{tsiotras}
K.~Okamoto, M.~Goldshtein, and P.~Tsiotras, ``Optimal covariance control for stochastic systems under chance constraints,'' {\em IEEE Control Systems Letters}, vol.~2, no.~2, pp.~266--271, 2018.

\bibitem{chen2021controlling}
Y.~Chen, T.~T. Georgiou, and M.~Pavon, ``Controlling uncertainty,'' {\em IEEE Control Systems Magazine}, vol.~41, no.~4, pp.~82--94, 2021.

\bibitem{chen2021optimal}
Y.~Chen, T.~T. Georgiou, and M.~Pavon, ``Optimal transport in systems and control,'' {\em Annual Review of Control, Robotics, and Autonomous Systems}, vol.~4, pp.~89--113, 2021.

\bibitem{leonard2014survey}
C.~Léonard, ``A survey of the {S}chr\"odinger problem and some of its connections with optimal transport,'' {\em Discrete and Continuous Dynamical Systems}, vol.~34, no.~4, pp.~1533--1574, 2014.

\bibitem{eldesoukey2023schr}
A.~Eldesoukey and T.~T. Georgiou, ``Schr\"odinger's control and estimation paradigm with spatio-temporal distributions on graphs,'' {\em arXiv preprint arXiv:2312.05679}, 2023.

\bibitem{miangolarra2024inferring}
O.~Movilla~Miangolarra, A.~Eldesoukey, and T.~T. Georgiou, ``Inferring potential landscapes: A {S}chr\"odinger bridge approach to {M}aximum {C}aliber,'' {\em arXiv preprint arXiv:2403.01357}, 2024.

\bibitem{de2021diffusion}
V.~De~Bortoli, J.~Thornton, J.~Heng, and A.~Doucet, ``Diffusion {S}chr{\"o}dinger bridge with applications to score-based generative modeling,'' {\em Advances in Neural Information Processing Systems}, vol.~34, pp.~17695--17709, 2021.

\bibitem{oksendal2013stochastic}
B.~{\O}ksendal, {\em Stochastic differential equations: an introduction with applications}.
\newblock Springer Science \& Business Media, 2013.

\bibitem{chen2022most}
Y.~Chen, T.~T. Georgiou, and M.~Pavon, ``The most likely evolution of diffusing and vanishing particles: {S}chr\"odinger bridges with unbalanced marginals,'' {\em SIAM Journal on Control and Optimization}, vol.~60, no.~4, pp.~2016--2039, 2022.

\bibitem{van2007stochastic}
R.~Van~Handel, ``Stochastic calculus, filtering, and stochastic control,'' {\em Course notes}, vol.~14, 2007.

\bibitem{rogers2000diffusions}
L.~C. Rogers and D.~Williams, {\em Diffusions, {M}arkov processes, and martingales: Volume 1, foundations}.
\newblock Cambridge University Press, 2000.

\bibitem{pollard1984skorohod}
D.~Pollard, ``The {S}korohod metric on ${D}[0,\infty)$,'' {\em Convergence of Stochastic Processes}, pp.~122--137, 1984.

\bibitem{evans2022partial}
L.~C. Evans, {\em Partial differential equations}, vol.~19.
\newblock American Mathematical Society, 2022.

\bibitem{leonard2022feynman}
C.~L{\'e}onard, ``Feynman-{K}ac formula under a finite entropy condition,'' {\em Probability Theory and Related Fields}, vol.~184, no.~3-4, pp.~1029--1091, 2022.

\bibitem{kullback1951information}
S.~Kullback and R.~A. Leibler, ``On information and sufficiency,'' {\em The Annals of Mathematical Statistics}, vol.~22, no.~1, pp.~79--86, 1951.

\bibitem{cover1999elements}
T.~M. Cover and J.~A. Thomas, {\em Elements of Information Theory 2nd Edition (Wiley Series in Telecommunications and Signal Processing)}.
\newblock Wiley-Interscience, 2006.

\bibitem{1978iii}
C.~Dellacherie and P.-A. Meyer, {\em Probabilities and Potential}, vol.~29 of {\em North-Holland Mathematics Studies}.
\newblock North-Holland, 1978.

\bibitem{chen2016entropic}
Y.~Chen, T.~T. Georgiou, and M.~Pavon, ``Entropic and displacement interpolation: a computational approach using the {H}ilbert metric,'' {\em SIAM Journal on Applied Mathematics}, vol.~76, no.~6, pp.~2375--2396, 2016.

\bibitem{georgiou2015positive}
T.~T. Georgiou and M.~Pavon, ``Positive contraction mappings for classical and quantum {S}chr{\"o}dinger systems,'' {\em Journal of Mathematical Physics}, vol.~56, no.~3, 2015.

\end{thebibliography}

\end{document}